\documentclass[ paper=a4, fontsize=11pt,bibliography=totoc, final]{scrartcl}
\usepackage{amsthm, amssymb, amsmath, amsfonts, mathrsfs, dsfont}

\usepackage[utf8]{inputenc}
\usepackage[T1]{fontenc}
\usepackage{lmodern}
\usepackage{graphicx,graphics}
\usepackage{mathtools}
\usepackage[english]{babel}
\usepackage{csquotes}
\usepackage{epsfig,url}  
\usepackage{bbm}
\usepackage{a4wide}
\usepackage{enumerate}
\usepackage{verbatim} 
\usepackage{color}
\usepackage{esint}
\usepackage{microtype}

\usepackage[T1]{fontenc} 

\usepackage[colorlinks=true, pdfstartview=FitV, linkcolor=blue, citecolor=blue, urlcolor=blue,pagebackref=false]{hyperref}
\usepackage[notcite,notref,color]{showkeys}
\definecolor{labelkey}{gray}{.8}
\definecolor{refkey}{gray}{.8}

\definecolor{darkblue}{rgb}{0,0,0.7} 
\definecolor{darkgreen}{rgb}{0,0.5,0}

\newcommand{\rw}[1]{{\color{darkblue}{#1}}}

\newcommand{\ud}{\;\mathrm{d}}

\providecommand{\eps}{\varepsilon}

\providecommand{\supp}{\operatorname{supp}}

\newtheorem{theorem}{Theorem}[section]
\newtheorem{corollary}[theorem]{Corollary}

\newtheorem{lemma}[theorem]{Lemma}

\newtheorem{assumption}[theorem]{Assumption}

\numberwithin{equation}{section}
\numberwithin{theorem}{section}

\newcommand{\heta}{\hat{\eta}}

\newcommand{\e}{\varepsilon}
\newcommand{\Reals}{{\mathbb R}}
\newcommand{\LL}{{\mathcal L}}
\newcommand{\KK}{{\mathcal K}}
\newcommand{\R}{{\mathbb R}}

\renewcommand{\varrho}{{\rho}}

\newcommand{\rmin}{r_{\operatorname{min}}}

\newcommand{\hv}{\hat{v}}

\newcommand{\step}[1]{\noindent \textit{Step} #1.}

\usepackage[margin=2.6cm,includefoot]{geometry}

\mathtoolsset{showonlyrefs}
\usepackage{authblk}

\setkomafont{date}{\large}

\title{The grazing collisions limit from the linearized Boltzmann equation to the Landau equation for short-range potentials}{}

\author[1]{Corentin Le Bihan \thanks{corentin.le-bihan@ens-lyon.fr}}
\author[2]{ Raphael Winter\thanks{raphael.elias.winter@univie.ac.at}}
\affil[1]{UMPA -- ENS de Lyon, France}
\affil[2]{University of Vienna, Austria}

\begin{document}
	
	\maketitle

	\begin{abstract}
			The Landau equation and the Boltzmann equation are connected through the limit of grazing collisions. This has been proved rigorously for certain families of Boltzmann  operators concentrating on grazing collisions. In this contribution, we study the collision kernels associated to the two-particle scattering via a finite range potential $\Phi(x)$ in three dimensions. We then consider the limit of weak interaction given by $\Phi_\e(x) = \e \Phi(x)$. Here $\e\rightarrow 0$ is the grazing parameter, and the rate of collisions is rescaled to obtain a non-trivial limit. The grazing collisions limit is of particular interest for potentials with a singularity of order $s\geq 0$ at the origin, i.e. $\phi(x) \sim |x|^{-s}$ as $|x|\rightarrow 0$. For $s\in [0,1]$,  we prove the convergence to the Landau equation with diffusion coefficient given by the Born approximation, as predicted in the works of Landau and Balescu.  On the other hand, for potentials with $s>1$ we obtain the non-cutoff Boltzmann equation in the limit. The Coulomb singularity $s=1$ appears as a threshold value with a logarithmic correction  to the diffusive timescale, the so-called \emph{Coulomb logarithm}.
	\end{abstract}

	\tableofcontents

\section{Introduction}

In this paper, we consider the grazing collisions limit for short-range potentials, with a focus  on singular potentials. The limit describes the collisional dynamics of a particle system with weak collisions, i.e. when the momentum exchanged in a single collision is small. Consider the Boltzmann equation in dimension three
\begin{align}\label{eq:BM1}
Q(f,f)(v_1) = \int_{\R^3} \int_{S^2} B(v_1-v_2,\sigma)(f(v_1')f(v_2') - f(v_1)f(v_2)) d\sigma dv_2,
\end{align}
where the collision kernel $B(v_1-v_2,\sigma)$ determines the rate of collisions with relative velocity $v_1-v_2$ and parameter $\sigma \in S^2$. More precisely, the in- and outgoing velocities are connected through the elastic collision rule
\begin{equation}\label{eq:sigma}
	\begin{aligned}
		v'_1&= \frac{v_1+v_2}{2} + \frac{|v_1-v_2|}{2}\sigma,\\
		v'_2&= \frac{v_1+v_2}{2} - \frac{|v_1-v_2|}{2}\sigma.
	\end{aligned}
\end{equation}
The cross-section $B$ can equivalently be expressed as
\begin{align}
B(v_1-v_2,\sigma) = \tilde{B}(v_1-v_2,\cos \theta),
\end{align}
where $\theta \in [0,\pi) $ is the deviation angle of the collision. In general, a sequence of Boltzmann operators $Q_\e$ describes grazing collisions, if the corresponding 
collision kernels $\tilde{B}_\e(v-v_*,\cos \theta)$ concentrate on angles $\theta$ close to $\tfrac\pi2$.  The family satisfies the grazing collisions limit if 
\begin{align} \label{eq:grazing}
Q_\e(f,g) \rightarrow Q_L(f,g) \quad \text{as } \e \rightarrow 0,
\end{align}
where $Q_L$ is the Landau collision operator
\begin{align}\label{eq:LandauOp}
    Q_L(f,g) = 2 \pi \nabla_v \cdot \left(\int_{\R^3} \frac{P_{v_1-v_2}^\perp}{|v_1-v_2|} (\nabla f(v_1) g(v_2) - f(v_1) \nabla g(v_2)) \ud{v_2}  \right). 
\end{align}
Here we denote by $P^\perp_{v_1-v_2}$ the orthogonal projection to $v_1-v_2\neq 0$, i.e.:
\begin{align}\label{orthproj} 
P^\perp_{v_1-v_2} x = x-\frac{(v_1-v_2)\cdot x}{|v_1-v_2|}\frac{v_1-v_2}{|v_1-v_2|}.
\end{align}
The Landau equation, and the grazing collisions limit were first introduced for  particle systems with Coulomb interaction in~\cite{landau_kinetische_1936}.  In the physics literature, the grazing limit has subsequently become a general principle for particle systems with weak interaction  (cf.~\cite{balescu_equilibrium_1975,lifshitz_course_1981}).
The grazing collisions limit can also be studied for the linearized Boltzmann- and Landau operators
\begin{align} \label{eq:linBoltzclassic}
    \LL_\e \psi &=   \int_{\R^3} \int_{S^2_+} (\psi(v'_{1})+\psi(v'_{2})-\psi(v_1)-\psi(v_2)) M(v_2) \tilde{B}_\e(v_1-v_2,\cos \theta )dv_2 d\sigma, \\
\label{eq:linLandau} 
\KK \psi (v_1) &= \frac{2\pi}{M(v_1)} \nabla_v\cdot  \left(\int_{\R^3}  \frac{P^\perp_{v_1-v_2}}{|v_1-v_2|} (\nabla \psi(v_1)-\nabla \psi(v_2)) M(v_1)M(v_2) \ud{v_2}\right),
\end{align}
when the challenge is to prove the convergence
\begin{align}
    \LL_\e \psi \rightarrow \KK \psi.
\end{align} 
Here  $M(v)$ denotes the Maxwellian equilibrium, without loss of generality with  unit temperature
\begin{align}\label{eq:Maxwellian}
M(v) = \frac{e^{-\frac12 |v|^2}}{(2\pi)^\frac32}.
\end{align}

% We remark that the can be studied for general collision kernels $\tilde{B}_\e(v-v_*,\cos \theta)$ which are not necessarily connected to a pair potential $\Phi$. 

In this contribution, we study the grazing collisions limit for Boltzmann equations which correspond to pairwise collisions of particles via a radially symmetric pair  potential $\Phi(x)$ with finite range, without loss of generality
\begin{align} \label{eq:range}
\supp \Phi \subset B_1(0).
\end{align}
We then consider the sequence of scaled potentials $\Phi_\e$, $\e\rightarrow 0$, given by
\begin{align}\label{eq:weak} 
\Phi_\e(x)=\e \Phi(x).
\end{align} 
This assumption of microscopic interaction~\eqref{eq:range} has proven to be technically convenient in the derivation of kinetic equations from scaling limits of interacting particle systems (cf.~\cite{gallagher_newton_2013,pulvirenti_validity_2014}). 
With the exception of the result in \cite{ayi_newtons_2017} for super-exponentially decaying potentials, the known derivations of the Boltzmann equation are restricted to microscopic interaction.
The Boltzmann equation and the grazing collisions limit can be used as an intermediate step in the derivation of the Landau equation from particle systems.  This motivates us to bridge the gap between the Boltzmann equation and the Landau equation for  microscopic interactions. So far, rigorous derivations of the Landau equations are restricted to test particle models, we refer to~\cite{basile_diffusion_2014,catapano_rigorous_2018,desvillettes_rigorous_2001,durr_asymptotic_1987,kesten_limit_1980}. For partial results on the derivation of the nonlinear Landau equation see~\cite{boblylev_particle_2013,winter_convergence_2021}. A generalization of the method in the present paper to non-compactly supported potentials might be possible, yet technically much more involved. 

Starting from~\cite{cercignani_boltzmann_1987}, the model case of power law potentials $\Phi(x) = |x|^{-(r-1)}$, $r\geq 2$ has been studied, for which the collision kernel factorizes as
\begin{align}\label{eq:factorizing}
\tilde{B}(v-v_*,\cos \theta) = |v-v_*|^\gamma b(\cos \theta), 
\end{align}
for some function $b$ and $\gamma= (r-5)/(r-1)$. Notice that for non-power law potentials, the collision kernel is not explicitly known and does not satisfy a similar factorization property. 

Since physically relevant interaction potentials are typically singular at the origin, we will study this case in detail here. More precisely, we study potentials $\Phi$ with
\begin{align}\label{eq:singular}
\Phi(x) \sim \frac{1}{|x|^s},  \quad \text{for } |x|\rightarrow 0,
\end{align}
where $s\geq 0$ is the order of the singularity.

% The representation~\eqref{eq:BM} of the Boltzmann equation allows us to formulate the grazing collision limit as a  rescaling of the interaction potential $\Phi$. More precisely, let $\Phi_\e$ be given by~\eqref{eq:weak} 
% and $Q_\e$ be given by
% \begin{align}
% Q_\e(f,g) := \mathfrak{d}_\e^{-1}   Q_{\Phi_\e}(f,g),
% \end{align}
% where $\mathfrak{d}_\e\rightarrow 0$ is an appropriately chosen sequence.

 It is conjectured that for $s\in [0,1]$ the Landau equation can be derived through the scaling~\eqref{eq:weak} on an appropriate timescale (e.g.~\cite{nota_interacting_2021-1,nota_interacting_2021}). We will give a rigorous proof for this conjecture in Theorem~\ref{thm:1}. In particular, this covers the physically important Coulomb singularity $s=1$, for which we obtain a logarithmic correction of the timescale, which is known as the \emph{Coulomb logarithm}. The Coulomb logarithm is ubiquitous in the plasma physics literature, but there are only few mathematically rigorous results on this notion. As such, the present paper gives a rigorous validation for the onset of the Coulomb logarithm in the grazing limit.  
 
% In this contribution, we rigorously prove the validity of~\eqref{eq:grazing} for the corresponding linearized operators. 

%  Let us now consider the linearizations of these operators. More precisely, for a given density $f$, let us introduce $\psi$ given by
% \begin{align}
% f(v) = M(v) (1+ \psi(v)).
% \end{align} 
% The linearized collision operators are  obtained by neglecting  quadratic terms in $\psi$.  The resulting linearized Landau operator $\KK$ can be expressed as
% \begin{align}
% \KK(\psi)(v_1) = \frac{2\pi}{M(v)} \nabla_v\cdot  \left(\int_{\R^3}  \frac{P^\perp_{v_1-v_2}}{|v_1-v_2|} (\nabla \psi(v_1)-\nabla \psi(v_2)) M(v_1)M(v_2)\right).
% \end{align}
% On the other hand, the linearized Boltzmann operator is given by
% \begin{align}
% \LL \psi (v_1)  =  \int_{\R^3} \int_{S^2_+} (\psi(v'_{1})+\psi(v'_{2})-\psi(v_1)-\psi(v_2)) M(v_2) (v_1-v_2)\cdot \nu dv_2 d\nu .
% \end{align}
% Here the outgoing velocities $v_1'$, $v_2'$ (cf.~\eqref{eq:outgoing}) are determined by the underlying collision dynamics. 

In Theorem~\ref{thm:2}, we prove that the Landau equation cannot be obtained for potentials of the form~\eqref{eq:singular} with $s>1$. Instead, we prove that for such potentials the scaling~\eqref{eq:weak} leads to the non-cutoff Boltzmann equation. The result is related to the derivation of the non-cutoff Boltzmann equation from truncations such as~\cite{desvillettes_linear_1999} and~\cite{duan_solutions_2021}. 

We achieve the proof of Theorems~\ref{thm:1} and Theorem~\ref{thm:2} by a careful analysis of the scattering of particles interacting through a singular potential. For the Lorentz gas and non-singular interaction, this has been studied in~\cite{desvillettes_rigorous_2001}, and some elements of the proofs are used in the present paper. However, the main point of our work is the precise analysis of the grazing limit for singular potentials $\Phi$, in particular around the threshold Coulomb singularity $s=1$. 

\paragraph{Previous results} 

The rigorous analysis of the grazing collisions limit goes back to the classical result of Arsenev and Buryak ~\cite{arsenev_connection_1991}. For the linearized equations, the limit has been verified in~\cite{desvillettes_asymptotics_1992} by directly scaling the Boltzmann collision kernels. A derivation of the Landau equation from the Boltzmann operator for Coulomb collisions can be found in~\cite{degond_fokker-planck_1992}. 

 A general framework for the grazing collisions limit was developed in the celebrated work~\cite{alexandre_landau_2004}. The result shows the validity of the grazing collisions limit under very natural and general, yet abstract assumptions. 
 
%  The assumptions of~\cite{alexandre_landau_2004} can be difficult to verify for a given limit. In particular, it is conjectured by the authors that truncations of the Debye cross-section, associated to Coulomb interaction, are covered by the result. However the authors write 'Even if all the tools seem there to perform this study, we were discouraged by the extremely tedious
% computations involved (although there does not seem to be any real conceptual difficulty). Therefore we did not
% check that the true Debye cross-section satisfies all our technical assumptions; for any interested reader we can
% provide the partial results that we obtained.' (p.75). 

A number of results on the grazing  limit have been obtained under the factorization assumption
\begin{align}
\tilde{B}_\e(v-v_*,\cos \theta) = |v-v_*|^\gamma b_\e(\cos \theta),
\end{align}
for cross-sections $b_\e$ which concentrate on grazing collisions.
We refer to~\cite{goudon_boltzmann_1997,he_asymptotic_2014} for results in this setting. In this framework, it has recently been shown  that the  gradient flow structures associated to the Boltzmann- and Landau equations are connected through the grazing collisions limit as well~\cite{carrillo_boltzmann_2022}.

We remark that the grazing limit is also used to connect the spectral analysis of the Boltzmann- and Landau operators (cf.~\cite{baranger_explicit_2005,mouhot_explicit_2006,mouhot_spectral_2007}).

% The grazing collision limit can be also studied on the level of the associated jump processes, see~\cite{guerin_convergence_2003}. 

\paragraph{Structure of the paper} 

The plan of the paper is as follows. In Section~\ref{sec:results} we present the main results of the paper and recall some basic facts on the linearized Boltzmann- and Landau operators. Section~\ref{sec:fokker} contains the proof of the grazing collisions limit for potentials $\Phi$ with singularity $\Phi(x)\sim |x|^{-s}$ near zero and $s\in[0,1]$. Finally, in Section~\ref{sec:hard}, we prove that for $s>1$ the limit of weak collisions leads to the non-cutoff Boltzmann equation.

\section{Preliminaries and main results}  \label{sec:results}

Let us recall some basic facts about the Boltzmann equation associated to pairwise collision of particles interacting through a short-range central potential $\Phi$ satisfying~\eqref{eq:range}. Consider the associated scattering problem and denote the velocities by $V_1(t),V_2(t)$ and the positions by $X_1(t),X_2(t)$. 
Without loss of generality, let us assume $|X_1(0)-X_2(0)|=1$. The equations of motion read
\begin{equation}\label{eq:scattering}
\begin{aligned}
\dot X_1(t) &= V_1(t), \quad &\dot V_1(t) &= - \nabla \Phi (X_1(t)-X_2(t)), \\
\dot X_1(t) &= V_1(t), \quad &\dot V_1(t) &= - \nabla \Phi (X_2(t)-X_1(t)),\\
X_1(0)-X_2(0)&= \nu \in S^2, \quad &V_i(0)& =v_i,  \quad i=1,2.
\end{aligned}
\end{equation} 
Here $\nu$ is the unit vector connecting the two particle centers when the particles are about to collide. In order to observe the collision forward in time we restrict to the set  $S^2_+$ defined by
\begin{align}
S^2_+ = \{v_2 \in \R^3: \nu \cdot (v_1-v_2)\geq 0\}.
\end{align}
The outgoing velocities $v_1'$, $v_2'$ are then given by 
\begin{align}\label{eq:outgoing} 
v_i'= \lim_{t\rightarrow \infty} V_i(t),\quad i=1,2,
\end{align}
whenever the limit in~\eqref{eq:outgoing} exists. 
The Boltzmann operator associated to a finite-range potential $\Phi$ allows for an equivalent representation in terms of the impact vector $\nu\in S^2$, originally introduced by King (cf.~\cite{king_bbgky_1975})
\begin{align}\label{eq:BM} 
Q_\Phi(f,g)(v_1) = \int_{\R^3} \int_{S^2_+} (v_1-v_2)\cdot \nu (f(v_1')g(v_2')-f(v_1)f(v_2)) dv_2 d\nu. 
\end{align}
In King's representation, the linearized Boltzmann operator ~\eqref{eq:linBoltzclassic} reads 
\begin{align} \label{eq:linBoltz}
\LL_\e \psi (v_1)  =  \int_{\R^3} \int_{S^2_+} (\psi(v'_{1,\e})+\psi(v'_{2,\e})-\psi(v_1)-\psi(v_2)) M(v_2) (v_1-v_2)\cdot \nu dv_2 d\nu ,
\end{align}
where $v_{1,\e}'$, $v_{2,\e}'$ are the outgoing velocities~\eqref{eq:outgoing} of the scattering problem~\eqref{eq:scattering} with interaction potential $\Phi_\e$ (cf. \eqref{eq:weak}) and $M(v)$ is the Maxwellian distribution introduced in~\eqref{eq:Maxwellian}. 

We remark that the operators $\LL_\e$~(cf.~\eqref{eq:linBoltz}) and $\KK$~(cf.~\eqref{eq:linLandau}) are self-adjoint on the natural space
\begin{align}\label{def:X} 
 L^2_M(\R^3) := L^2 (M(v)  dv), 
\end{align}
with   inner product 
\begin{align}\label{def:inner}
(\psi,\zeta)_{ L^2_M(\R^3)} := \int_{\R^3}\psi(v) \zeta(v )M (v)  dv.
\end{align}
Let us similarly recall the space
\begin{align}
   L^2_M(\R^3\times \R^3) := L^2 (M(v)  dv dx).
\end{align}

We assume $\Phi$ to be a short-range potential with a singularity of order $s\in [0,\infty)$ at the origin. More precisely, we make the following assumption on $\Phi$.

\begin{assumption}[Interaction Potential] \label{ass:potential} 
	Let $\Phi\in C(\R^3\setminus\{0\})$ be a radially symmetric potential, i.e. 
	\begin{align}
	\Phi(x) = \phi(|x|),
	\end{align}
	and further assume that $\phi$ is decreasing. Moreover, let $\phi$ be of the form
	\begin{align}
	\phi(r) = \begin{cases} \frac{f(r)}{|r|^s} \quad &\text{for $0<r\leq 1$}, \\
	0\quad &else,\end{cases}
	\end{align}
	for some $s\geq 0$ and a decreasing function $f\in C_b^\infty \left(0,1\right)$  with $f(0)>0$, and $f(1)=0$.
\end{assumption}
For future reference, let us introduce the function $K(\rho)$, $\rho>0$, defined by 
\begin{align}\label{def:K}
K(\rho) = \sup_{\rho<r<1} \{|\phi(r)|,r|\phi'(r)|,r^2 |\phi''(r)|\}.
\end{align}
We are now in the position to state the first main result.
\begin{theorem}[Grazing collisions limit ]\label{thm:1} 
	Let $\Phi$ be an interaction potential which satisfies Assumption~\ref{ass:potential} for some $ s \in [0,1]$. 
	Let $\mathfrak{d}_\e$ be given by
	\begin{align}\label{eq:muEps}
		  \mathfrak{d}_\e = \begin{cases} \e^{2} \quad & \text{if $s\in [0,1)$}, \\
			\e^{2} |\log \e| \quad &\text{if $s=1$}  .
			\end{cases} 
	\end{align}
	Then there exists a constant $c_\Phi>0$ such that
	\begin{align}
		\| \mathfrak{d}_\e^{-1} \LL_\e \psi - 2\pi  c_\Phi\KK \psi \|_{L^\infty(\R^3)} \leq C |\log(\e)|^{-1} \|\psi\|_{C^3_b(\R^3)},
	\end{align}
	where $\mathcal{K}$ is the operator defined in~\eqref{eq:linLandau}.
	The constant $c_\Phi>0$ can be explicitly expressed as
	\begin{align} \label{eq:cPhi} 
		c_\Phi = \begin{cases} \int_0^1 \big(\int_0^1\frac{\rho}{u}\phi'\big(\frac{\rho}{u}\big)\frac{du}{\sqrt{1-u^2}}\big)^2\rho\,d\rho\quad &\text{if $s\in [0,1)$,}\\
		f(0)&\text{if $s=1$.}\end{cases}
	\end{align}
	 The constant $c_\Phi$ introduced in~\eqref{eq:cPhi} satisfies the identity
    \begin{align}\label{equiv:cPhi}
        2\pi c_\Phi = \frac1{8\pi}\int_{\Reals^3} \delta(k\cdot e_1) |k|^2 |\hat{\Phi}(k)|^2 \ud{k},
    \end{align}
    where we use the convention $\hat{\Phi}(k)= \int_{\Reals^3} e^{-ik\cdot x} \Phi(x)\ud{x}$ for the Fourier transform of $\Phi$.
\end{theorem} 
The theorem will be proved in the core of the paper in Subsection~\ref{subsec:Thm1}. Let emphasize that the scaling of $\mathfrak{d}_\e$ in~\eqref{eq:muEps} reveals the onset of the Coulomb logarithm for $s=1$.  

It is important to remark that the constant $c_\Phi$ coincides with the expression derived in other works, such as~\cite{boblylev_particle_2013,pulvirenti_propagation_2016,nota_interacting_2021-1,nota_interacting_2021,winter_convergence_2021}

As a consequence of Theorem~\ref{thm:1}, we obtain the convergence for the solutions $\psi_\e$ of the linearized Boltzmann equation, to the solution $\psi$ of the linearized Landau equation.

\begin{corollary}
    \label{cor:solutions}
	Consider $\LL_\e$, $\e>0$ the family of linearized Boltzmann operators in Theorem~\ref{thm:1}, and let $\psi_\e$ be the family of solutions
	\begin{align}
	\partial_t\psi_\e  = \mathfrak{d}_\e^{-1} \LL_\e \psi_\e, \quad \psi_\e(0) = \psi^\circ.  
	\end{align} 
	Then the sequence $\psi_\e \in L^\infty(\R^+ ; L^2_M(\R^3 \times \R^3))$ is uniformly bounded, and
	\begin{align}
	\psi_\e \rightharpoonup^* \psi \quad \text{in } L^\infty(\R^+ ; L^2_M(\R^3 \times \R^3)),
	\end{align}
	where $\psi$ solves the linearized Landau equation
	\begin{align}
	\partial_t \psi + v\cdot \nabla_x \psi = \KK\psi, \quad \psi(0) = \psi^\circ. 
	\end{align} 
\end{corollary}
The proof of this corollary is rather standard, and can be found in  Subsection~\ref{subsec:solutions}. 

We now discuss the case of short-range potentials $\Phi$ satisfying Assumption~\ref{ass:potential} with $s>1$. We still obtain a limiting operator $\mathcal{L}_\infty$ for the scaling~\eqref{eq:weak}. However, the limit operator is not the linearized Landau operator. Instead, the family of Boltzmann operators given by the scaling~\eqref{eq:weak} converges to the linearized non-cutoff Boltzmann operator.

More precisely, we obtain the following theorem.
\begin{theorem} \label{thm:2}
	Let $\Phi$ be an interaction potential satisfying Assumption~ \ref{ass:potential} with $s>1$.
	Then for $\psi\in C^3_B(\R^3)$ we have 
	\begin{align} 
	\lim_{\e \rightarrow 0} \frac{1}{\e^\frac2s}	\LL_\e \psi = \LL_\infty \psi,
	\end{align}
	where $\LL_\infty$ is the non-cutoff Boltzmann operator given by
	\begin{align}
	\LL_\infty \psi =   \int_{\R^3} \int_{S^2_+} (\psi(v'_{1})+\psi(v'_{2})-\psi(v_1)-\psi(v_2)) M(v_2) \tilde{B}_s(v_1-v_2,\cos \theta )dv_2 d\sigma ,
	\end{align}
	where  $\tilde{B}_s(v_1-v_2,\cos \theta)$ is the collision kernel of the homogeneous potential
	\begin{align}
	\Phi_s (x) = \frac{1}{|x|^s}. 
	\end{align}
\end{theorem}
The proof of this Theorem is given in Section~\ref{sec:hard}.

We observe that the Coulomb singularity with $s=1$ appears as a threshold value, determining the regions in which Theorem~\ref{thm:1} and Theorem~\ref{thm:2} are valid. Intuitively, this can be understood as a competition between 
\begin{enumerate}[(a)]
	\item Boltzmann-like collisions with small impact parameter $\rho$, i.e. $\rho\lesssim l_{\operatorname{coll}} $, where $	l_{\operatorname{coll}}$ is the collision length
	\begin{align}
	l_{\operatorname{coll}} = \e^\frac1s.
	\end{align}
	A single collision of this type yields a deviation of order one in the particle velocity.  \label{type:BM}
	\item weak collisions with impact parameter $\rho\gg l_{\operatorname{coll}}$. Such collisions yield small deviations with mean zero and variance $\e^2$.   \label{type:Landau}
\end{enumerate}
The frequency of collisions of type~\eqref{type:BM} is $\e^\frac2s$, the frequency of collisions of type~\eqref{type:Landau} remains unscaled. By comparing the dominant contribution to the variance, we can identify $s=1$ as the threshold value. At the threshold value, a detailed analysis including logarithmic corrections is required to prove that weak collisions of type~\eqref{type:Landau} are dominant. This is a key point of the present paper.

\section{Grazing collisions limit for short-range potentials with \texorpdfstring{$s<1$}{s < 1}} \label{sec:fokker}
\subsection{Estimates for the two-particle scattering}\label{subsec:scattering}

In this section we give detailed estimates for the collision kernel of the Boltzmann equation based on the two-particle scattering. We start by introducing some definitions. Given incoming velocities $v_1\neq v_2\in \R^3$, we denote by $\eta,\eta'$ the modulus of the incoming/outgoing relative velocities 
\begin{align}\label{def:eta}
\eta= \frac{v_1-v_2}{|v_1-v_2|},\quad \eta'= \frac{v_1'-v_2'}{|v_1'-v_2'|}.
\end{align}
We also introduce the angle deviation $\heta$ and the deflection $\hv$ by
\begin{align} \label{eq:hats} 
\heta= \eta'-\eta, \quad \hv = v_1'-v_1.
\end{align}
For future reference, we remark that the following identity holds 
\begin{align}\label{eq:elementary}
\hv = \frac{|v_1-v_2|}{2} \heta.
\end{align}
Let  $v_1',v_2'$ be the outgoing velocities given by \eqref{eq:sigma} and some $\sigma \in S^2$, then the deviation angle $\theta$ is given by
\begin{align}\label{eq:costheta}
\cos \theta= \eta \cdot \sigma .
\end{align}
For $\nu\in S^2_+$ as in \eqref{eq:scattering}-\eqref{eq:outgoing},  let the impact factor $\rho>0$ and $\eta^\perp$ be given by the identities 
\begin{equation} \label{eq:rho}
	\begin{aligned}
		\rho 					&= |P_{\eta}^\perp \nu|,\\
		\rho \cdot  \eta^\perp  &= P_{\eta}^\perp \nu. 
	\end{aligned}
\end{equation}
Notice that $\eta^\perp\in S^1_\eta :=\{v\in \R^3: |v|=1,\,v\cdot \eta = 0 \}$.
Since we only consider rotationally symmetric potentials $\Phi$, we can choose without loss of generality an orthonormal frame such that $\eta= e_1$ and $\eta^\perp  =e_2$, i.e.  
\begin{align}\label{eq:nuVariable}
	\nu = (\sqrt{1-\rho^2},\rho, 0), \quad \nu\cdot (v_1-v_2)d\nu = |v_1-v_2|\rho\,d\rho\,d\eta^\perp.
	\end{align}
Using the impact parameter $\rho$ as integration variable, the linearized Boltzmann operator $\LL_\e$ takes the form
\begin{align}\label{eq:linBoltzImp}
	\LL_\e \psi (v_1)  = \int_{\R^3} \int_{S^1_\eta} \int_0^1 (\psi(v'_{1,\e})+\psi(v'_{2,\e})-\psi(v_1)-\psi(v_2))  \rho M(v_2)|v_1-v_2|  d\rho\,d\eta^\perp dv_2  .
	\end{align}
Our estimates of the family of operators introduced in~\eqref{eq:linBoltzImp} rely on the following representation of the scattering angle $\theta$ (cf. \eqref{eq:costheta})
\begin{align}\label{eq:theta}
\theta = \pi -2  \int_{\rmin}^\infty \frac{\rho}{r^2\sqrt{1-\tfrac{\rho^2}{r^2}-\tfrac{4\Phi_\e(r)}{|v_1-v_2|^2} }} dr ,
\end{align}
which can be found for example in \cite{gallagher_newton_2013}, section 8.3. Here $\rmin$ is the minimal distance of the particles 
\begin{align}
\rmin := \inf_{t\in[0,\infty)} |X_1(t)-X_2(t)|,
\end{align}
which satisfies the identity
\begin{align}\label{eq:rminID}
1-\frac{2\e \phi(\rmin)}{|v_1-v_2|^2}- \frac{\rho^2}{\rmin^2} = 0.
\end{align}

First we establish the following asymptotics for the angle $\theta$ (cf.~\eqref{eq:theta}):

\begin{lemma}\label{lem:central}
	Let $\kappa=\frac{2\e}{|v_1-v_2|^2}$. Then
	\begin{align}\label{eq:Kasym}
	\left|\theta(\rho) + 2 \kappa \int_0^1 \tfrac{\rho}{u} \phi'(\tfrac{\rho}{u}) \frac{du}{\sqrt{1-u^2}}\right| \leq C \kappa^2 K^2(\rho), \quad \text{for $\kappa <\frac{\rho^s}{4f(0)}$}
	\end{align}
	where $K(\rho)$ as defined in~\eqref{def:K}. 
\end{lemma}

This estimate is valid when ${\e f(0)}/{\rho^s}\leq |v_1-v_2|^2/4$, i.e. when the kinetic energy is higher than the potential energy. Note that in the following proof we use only that $\phi$ is a decreasing potential so the proof is valid for potentials with singularity $s>1$ and with non-compact support.
\begin{proof}
	This step is an adaptation of the proof of Proposition~1 in \cite{desvillettes_rigorous_2001}, for the case of singular potentials. 
	We change variables to the implicit variable $u$ given by
	\begin{align}\label{eq:uchange}
	u = \frac{\rho}{r \sqrt{1-2\kappa \phi(r)}}, \quad du = -\frac{\rho}{r^2\sqrt{1-2\kappa\phi(r)}}  \left( 1- \frac{r\kappa\phi'(r)}{1-2\kappa \phi(r)} \right) dr.
	\end{align}
	%	Notice that we have the identity
	%	\begin{align}
	%		\frac{\sqrt{1-\tfrac{p^2}{r^2}-2\kappa \phi(r)}}{\sqrt{1-2\kappa\phi}} = \sqrt{1-u^2}.
	%	\end{align}
	Using this we obtain
	\begin{equation} \label{eq:thetamanip} 
	\begin{aligned}
	\theta &= \pi -2  \int_{\rmin}^\infty \frac{\rho}{r^2\sqrt{1-\tfrac{\rho^2}{r^2}-\tfrac{4\Phi_\e(r)}{|v_1-v_2|^2} }} dr \\
	&= \pi -2  \int_{0}^1 \frac{1-2\kappa \phi(r) }{ (1-2\kappa \phi(r)- r\kappa \phi'(r))\sqrt{1-u^2}}du \\
	&= 2  \int_{0}^1 1-\frac{1-2\kappa \phi(r)}{ (1-2\kappa \phi(r)- r\kappa \phi'(r))\sqrt{1-u^2}}du\\
	&= 2  \int_{0}^1 \frac{r\kappa \phi'(r)}{ (1-2\kappa \phi(r)- r\kappa \phi'(r))\sqrt{1-u^2}}du.
	\end{aligned}  
	\end{equation}
	The integrand can be expanded as
	\begin{equation} \label{eq:expansion2} 
	\begin{aligned}
	\frac{r\kappa \phi'(r)}{ (1-2\kappa \phi(r)- r\kappa \phi'(r))} 
	= &\kappa \tfrac{\rho}{u} \phi'(\tfrac{\rho}{u}) + \kappa (r\phi'(r)- \tfrac{\rho}{u} \phi'(\tfrac{\rho}{u}) )  \\
	+&\kappa^2 \frac{\phi'(r)(2\phi(r)+r\phi'(r))}{1-2\kappa \phi(r)- r\kappa \phi'(r)}
	\end{aligned}
	\end{equation}
	For $\kappa<\frac{\rho^s}{4f(0)}$, we can bound the second and third term of the expansion~\eqref{eq:expansion2} by
	\begin{align*}
		\bigg| \kappa &(r\phi'(r)- \tfrac{\rho}{u} \phi'(\tfrac{\rho}{u}) ) +\kappa^2 \frac{\phi'(r)(2\phi(r)+r\phi'(r))}{1-2\kappa \phi(r)- r\kappa \phi'(r)}\bigg| \\
		&\leq  	 \kappa |r-\tfrac{\rho}{u}| \sup_{w\in \big(\tfrac{\rho}{u},r\big)} (|\phi'(w)|+w|\phi''(w)|) + \kappa^2 C K^2(\rho)
		\end{align*}
	Since $\phi(r)\geq0$ and $u\in (0,1)$, the change of variables~\eqref{eq:uchange} satisfies $r\geq \rho$. Therefore we have
	\begin{align}
		|r-\tfrac{\rho}{u}|=\Big|\frac{\rho}{u}\Big(1-\frac{1}{\sqrt{1-\kappa\phi(r)}}\Big)\Big| \leq \frac{C\rho}{u} |\kappa \phi(r)|\leq \frac{C \kappa K(\rho) \rho}{u}.
		\end{align}
	Inserting this above, we obtain 
	\begin{align}
	    \bigg| \kappa (r\phi'(r)- \tfrac{\rho}{u} \phi'(\tfrac{\rho}{u}) ) +\kappa^2 \frac{\phi'(r)(2\phi(r)+r\phi'(r))}{1-2\kappa \phi(r)- r\kappa \phi'(r)}\bigg| \leq  \kappa^2 C K^2(\rho).
	\end{align}
	In combination with the expansion~\eqref{eq:expansion2}
	this yields the desired estimate~\eqref{eq:Kasym}. 
\end{proof}
In the following two corollaries, we give integral estimates which can be derived from Lemma~\ref{lem:central}.
\begin{corollary}\label{cor:sinuses} 
    Let $\Phi$ satisfy Assumption~ \ref{ass:potential}, $s\geq 0$ and recall the notation $\kappa = \frac{2\eps}{|v_1-v_2|^2}$. Then the following estimates hold:  
    \begin{enumerate}
        \item For $s\in [0,1)$, we have
            \begin{equation} \label{eq:alpha0}
            	\begin{aligned}
	        \left|\int_0^1 \sin^2(\tfrac{\theta}{2})\,\rho\,d\rho -\kappa^2 c_\Phi  \right|&\leq C \kappa^{\frac{2}{s}},\\
	        \int_0^1 |\theta|^3\,\rho\,d\rho   &\leq C \kappa^{\frac{2}{s}} ,
	\end{aligned}
	\end{equation} 
	where $c_\Phi$ is given by~\eqref{eq:cPhi}.
	\item Similarly, for $s=1$ we can estimate
	\begin{equation} \label{eq:alpha1} 
	\begin{aligned}
	\left|\int_0^1 \sin^2(\tfrac{\theta}{2})\,\rho\,d\rho - \kappa^2 |\log \kappa|  c_\Phi  \right|&\leq C \kappa^2,\\
	\rw{\int_0^1 |\theta|^3}\rho\,d\rho  &\leq C \kappa^2.
	\end{aligned}
	\end{equation} 
    \end{enumerate}
\end{corollary}
\begin{proof}	
    We give the proof of~\eqref{eq:alpha1}, the adaptation to the case $s\in [0,1)$ is straightforward.
    
	We observe that there exists a constant $C>0$ such that
	\begin{align}
	K(\rho) &\leq \frac{C}{\rho},
	\end{align}
	which we use to infer the estimate 
	\begin{align}\label{eq:step1} 
	\left| \kappa \int_\rho^1\tfrac{\rho}{u}  \phi'(\tfrac{\rho}{u}) \frac{du}{\sqrt{1-u^2 }}\right| &\leq 	\left| \kappa \int_\rho^1 K(\rho) \frac{du}{\sqrt{1-u^2 }}\right| \leq \frac{C\kappa}{\rho }.
	\end{align}
	Next, we observe that by assumption on $\Phi$
	\begin{align*}
	r\phi'(r) = - \frac{f(r)}{r} + g(r),
	\end{align*}
	for some continuous, bounded function $g$. Therefore, for some constant $C>0$ we can estimate
	\begin{align*}
	\left| \int_\rho^1 \frac{\rho}{u} \phi'(\tfrac{\rho}{u})\frac{du}{\sqrt{1-u^2}}  + \frac{f(0)}{\rho} \right| \leq C.
	\end{align*}
	For $\rho>4f(0)\kappa$, we can apply~\eqref{eq:Kasym} and~\eqref{eq:step1} to obtain 
	\begin{align} \label{eq:step2} 
	    \Big|\frac{\theta(\rho)}{2}-\frac{\kappa f(0)}{\rho}\Big|&\leq C\Big(\kappa+\frac{\kappa^2}{\rho^2}\Big), \\
	    |\theta(\rho)|&\leq C\frac{\kappa}{\rho}.
	\end{align}
    We use these estimates to bound 
	\begin{align*}\Big|\sin^2\big(\tfrac{\theta}{2}\big)-\tfrac{f(0)^2}{\rho^2}\Big|&\leq \Big|\sin^2\big(\tfrac{\theta}{2}\big)-\big(\tfrac{\theta}{2}\big)^2\Big| + \Big|\big(\tfrac{\theta}{2}\big)^2-\big(\tfrac{\kappa f(0)}{\rho}\big)^2\Big|\\
	&\leq C|\theta|^3+C\tfrac{\kappa}{\rho}\big(\kappa+\tfrac{\kappa^2}{\rho^2}\big)\\
	&\leq C\big(\tfrac{\kappa^2}{\rho}+\tfrac{\kappa^3}{\rho^3}\big), \end{align*}
	for $\rho>4f(0)\kappa$.
    We now separate the regions $\rho>4f(0)\kappa$ and $\rho<4f(0)\kappa$ and conclude
	\begin{align*}
	\left|\int_0^1  \sin^2(\tfrac{\theta}2) \rho d\rho  -\kappa^2 |\log 4f(0)\kappa| f^2(0) \right|&\leq \int_0^{4f(0)\kappa}\rho\,d\rho+\int_{4f(0)\kappa}^1\Big|\sin^2\big(\tfrac{\theta}{2}\big)-\tfrac{f(0)^2}{\rho^2}\Big|\rho\,d\rho\\
	&\leq C\kappa^2 +\int_{4f(0)\kappa}^1C\big(\tfrac{\kappa^2}{\rho}+\tfrac{\kappa^3}{\rho^3}\big)\rho\,d\rho\\
	&\leq C\kappa^2,
	\end{align*}
	which proves the first assertion in~\eqref{eq:alpha1}. The second assertion follows immediately from~\eqref{eq:step2}. 
\end{proof}
\begin{corollary}
	For $s\in (0,1)$, we have the bound
	\begin{equation} \label{eq:a0}
	\begin{aligned}
	\left| \int_{S^2_+} \hv (v_1-v_2)\cdot \nu d\nu + \frac{8 \pi c_\phi \e^2 }{|v_1-v_2|^2} \eta \right| &\leq C \frac{\e^\frac2s }{|v_1-v_2|^{\frac4s -2}}  \\
	\left| \int_{S^2_+} (\hv \otimes \hv) (v_1-v_2)\cdot \nu d\nu - \frac{4\pi c_\phi  \e^2 }{|v_1-v_2|} P_\eta^\perp  \right| &\leq C\frac{\e}{|v_1-v_2|^{\frac4s-3}}  \\
	\int_{S^2_+} |\hv|^3 |(v_1-v_2)\cdot \nu| d\nu  &\leq C \frac{4\e^\frac2s}{|v_1-v_2|^{\frac4s -4}}. 
	\end{aligned}
	\end{equation} 
	
	For $s=1$, and $|v_1-v_2|\geq (2\e)^\frac12$ the following estimate holds:
	\begin{equation}
		\begin{aligned}\label{eq:a1}
		\left| \int_{S^2_+} \hv (v_1-v_2)\cdot \nu d\nu + \frac{8 \pi c_\Phi\e^2 |\log \e|}{|v_1-v_2|^2} \eta \right| &\leq C \frac{\e^2 |\log(v_1-v_2)|}{|v_1-v_2|^2}  \\
		\left| \int_{S^2_+} (\hv \otimes \hv) (v_1-v_2)\cdot \nu d\nu - \frac{4 \pi c_\Phi \e^2 |\log \e|}{|v_1-v_2|} P_\eta^\perp  \right| &\leq C \frac{\e^2 |\log(v_1-v_2)|}{|v_1-v_2|} \\
		\int_{S^2_+} |\hv|^3 |(v_1-v_2)\cdot \nu| d\nu  &\leq C \e^2.
		\end{aligned}
	\end{equation}
\end{corollary}
\begin{proof}
	We will only demonstrate the proof of~\eqref{eq:a0}, the proof of~\eqref{eq:a1} is analogous. 
	
	By definition of $\theta$ (cf.~\eqref{eq:costheta}), the collision rule~\eqref{eq:costheta} and definition of $\eta^\perp$~\eqref{eq:nuVariable} we can rewrite $\hat{v}$ as
	\begin{align}\label{eq:hatv}
		\hat{v} = v_1' - v_1 = \frac{|v_1-v_2|}{2} \left( (\cos \theta -1)\eta + \sin (\theta) \eta^\perp\right).
	\end{align}
	Furthermore, we recall the change of variables~\eqref{eq:nuVariable} and compute 
	\begin{align}
		\int_{S^2_+} \hv (v_1-v_2)\cdot \nu d\nu &= \int_{S^1_\eta}\int_0^1 \rho \frac{|v_1-v_2|^2}{2}  \left( (\cos \theta -1)\eta + \sin (\theta) \eta^\perp\right) \ud{\rho}\ud{\eta^\perp }\\
		&= - 2\pi \int_0^1 \rho |v_1-v_2|^2  \sin^2(\tfrac{\theta}2)\eta  \ud{\rho}
	\end{align}
	The first line in~\eqref{eq:a0} now follows immediately from~\eqref{eq:alpha0}.  
	
	For the proof of the second line in~\eqref{eq:a0} we again use~\eqref{eq:hatv} and find
	\begin{align}
	     &\int_{S^2_+} (\hv \otimes \hv) (v_1-v_2)\cdot \nu d\nu\\
	     = &\int_0^1 \int_{S^1_\eta} \left( \frac{|v_1-v_2|}{2} \left( (\cos \theta -1)\eta + \sin (\theta) \eta^\perp\right)\right)^{\otimes 2}  |v_1-v_2| \rho \ud{\rho} \ud{\eta^\perp} \\
	     = &\int_0^1 \int_{S^1_\eta}   \left( \sin^2(\tfrac{\theta}{2})\eta + \frac{\sin (\theta)}{2} \eta^\perp\right)^{\otimes 2}  |v_1-v_2|^2 \rho \ud{\rho}\ud{\eta^\perp } \\
	     = &\int_0^1 \int_{S^1_\eta} (\sin^2(\tfrac{\theta}{2})\eta)^{\otimes 2} + (\frac{\sin (\theta)}{2} \eta^\perp)^{\otimes 2}  |v_1-v_2|^2 \rho \ud{\rho}\ud{\eta^\perp}\\
        = &\int_0^1 \int_{S^1_\eta} \big(\sin^4(\tfrac{\theta}{2})\eta^{\otimes 2} +  \sin^2(\tfrac{\theta}{2})\cos^2(\tfrac{\theta}{2})(\eta^\perp)^{\otimes 2}\big)  |v_1-v_2|^2 \rho \ud{\rho}\ud{\eta^\perp}.
	\end{align}
	Now the claim follows from~\eqref{eq:alpha0} and the identity
	\begin{align}
	    \int_{S^1_\eta } \eta^\perp \otimes \eta^\perp \ud{\eta^\perp} &= \pi P_\eta^\perp 
	\end{align}
	The third line in~\eqref{eq:a0} follows along the same lines using~\eqref{eq:alpha0}. 
\end{proof}

\subsection{Proof of Theorem~\ref{thm:1}} \label{subsec:Thm1}
We are now in the position to prove Theorem~\ref{thm:1}. 
\begin{proof}[Proof of Theorem~\ref{thm:1}]
	Let $\psi \in C^3_b(\R^3)$ be arbitrary, and $v_1\in \R^3$. We then expand $\psi$ to find 
	\begin{equation}\label{eq:expansion}
	\begin{aligned}
	\LL_\e \psi (v_1)  = &\frac12 \int_{\R^3} \int_{S^2_+} (\psi(v'_{1,\e})+\psi(v'_{2,\e})-\psi(v_1)-\psi(v_2)) M(v_2) (v_1-v_2)\cdot \nu dv_2 d\nu \\
	= &I_1(v)+I_2(v)+R(v_1),
	\end{aligned}
	\end{equation} 
	where $I_1$ and $I_2$ are given by
	\begin{align}
	I_1(v)=&\frac12 \int_{\R^3} \int_{S^2_+} (\hv \cdot \nabla \psi(v_1)+\frac12  \nabla^2 \psi(v_1) : [\hv\otimes \hv ] ) M(v_2) (v_1-v_2)\cdot \nu dv_2 d\nu \\
	I_2(v) =& -\frac12 \int_{\R^3} \int_{S^2_+} (\hv \cdot \nabla \psi(v_2)+\frac12  \nabla^2 \psi(v_2) : [\hv\otimes \hv ] ) M(v_2) (v_1-v_2)\cdot \nu dv_2 d\nu,
	\end{align}
	and $R(v_1)$ is a remainder which can be estimated by
	\begin{align}
	|R(v_1)|\leq C_\psi \int_{\R^3} \int_{S^2_+} |\hv|^3 M(v_2) |(v_1-v_2)\cdot \nu| dv_2 d\nu. 
	\end{align}
	We now distinguish the case $s=1$ from $s\in [0,1)$.
	
	\step1 {$s=1$:}
	
	From~\eqref{eq:a1} we obtain 
	\begin{align}
	|R(v_1)| \leq C_\psi \e^2. 
	\end{align}
	It remains to determine the limit of $I_1$ and $I_2$. Using~\eqref{eq:a1} we get
	\begin{align}
	&\left| I_1(v_1)+2\pi \e^2 |\log \e| c_\Phi \int_{\R^3} \frac{4\eta \nabla \psi(v_1) } {|v_1-v_2|} -\frac{P_\eta^\perp}{|v_1-v_2|}: \nabla^2 \psi(v)  M(v_2 )dv_2\right| \\
	\leq C_\psi&\int_{\{v_2:|v_1-v_2|\geq \e^\frac12\}}   \left(\frac{\e^2 |\log (|v_1-v_2|)|}{|v_1-v_2|^2} +\frac{\e^2 |\log (|v_1-v_2|)|}{|v_1-v_2|} \right) M(v_2) dv_2  \\
	+& \int_{\{v_2:|v_1-v_2|\leq \e^\frac12\}}  |\hv| M(v_2) dv_2
	\\
	\leq C_\psi  &\e^2.
	\end{align} 
	Notice that in the last inequality we have made use of $|\hv|\leq |v_1-v_2|$, which follows from~\eqref{eq:elementary}. Performing the same estimate for $I_2$ we obtain
	\begin{align}
	&\left| \LL_\e \psi (v_1) - 2\pi \e^2 |\log \e| c_\Phi \LL_0\psi (v_1)\right| \leq C_\psi \e^2, \quad \text{where}\\
	&\LL_0\psi = \int_{\R^3} \left(\frac{4 \eta \cdot  (\nabla \psi(v_2)-\nabla \psi(v_1))}{|v_1-v_2|^2}  + \frac{P_\eta^\perp: (\nabla^2 \psi(v_1)+\nabla^2 \psi(v_2))}{|v_1-v_2|} \right) M(v_2) dv_2 .
	\end{align}
	
	\step3 It remains to show that $\LL_0$ coincides with the desired Landau operator.
	
	To this end, we first remark that 
	\begin{align}\label{eq:matrixId}
	\nabla \cdot \frac{ P_\eta^\perp}{|v_1-v_2|} = -2 \frac{v_1-v_2}{|v_1-v_2|^3}.
	\end{align}
	This allows us to rewrite $\LL_0$ as (writing $\mathcal{A}=\tfrac{ P_\eta^\perp}{|v_1-v_2|}$ ) 
	\begin{align*}
	\LL_0\psi = &\int_{\R^3} \left(\frac{4 \eta \cdot  (\nabla \psi(v_2)-\nabla \psi(v_1))}{|v_1-v_2|^2}  + \frac{P_\eta^\perp: (\nabla^2 \psi(v_1)+\nabla^2 \psi(v_2))}{|v_1-v_2|} \right) M(v_2) dv_2  \\
	= &\int_{\R^3}\left(  (\nabla_{v_1}-\nabla_{v_2})(\mathcal{A}) (\nabla \psi(v_1)-\nabla \psi(v_2)) +\mathcal{A} (\nabla^2 \psi(v_1)+\nabla^2 \psi(v_2))\right)  M(v_2) dv_2 \\
	= &\frac{1}{M(v_1)}\int_{\R^3}  (\nabla_{v_1}-\nabla_{v_2})\left[\mathcal{A} (\nabla \psi(v_1)-\nabla \psi(v_2)) +\mathcal{A} (\nabla^2 \psi(v_1)\right] M(v_1)M(v_2) dv_2,
	\end{align*}
	which finally leads to
	\begin{align*}
	\LL_0\psi 	= &\frac{1}{M(v_1)} \nabla_{v_1} \cdot \left( \int_{\R^3}  \mathcal{A} (\nabla \psi(v_1)-\nabla \psi(v_2)) +\mathcal{A} (\nabla^2 \psi(v_1) M(v_1)M(v_2) dv_2 \right) \\
	-&\frac{1}{M(v_1)}\int_{\R^3} \left(\mathcal{A} (\nabla \psi(v_1)-\nabla \psi(v_2)) +\mathcal{A} (\nabla^2 \psi(v_1)\right) (\nabla_{v_1}-\nabla_{v_2})\left[M(v_1)M(v_2) \right] dv_2.
	\end{align*}
	The first line coincides with $\KK \psi$. 
	It now remains to remark that the last line vanishes
	\begin{align*}
	\mathcal{A} \cdot (\nabla_{v_1}-\nabla_{v_2})\left[M(v_1)M(v_2) \right] = 0,
	\end{align*}
	hence $\LL_0 = \KK$.
	
	\step4 The case $s \in [0,1)$ follows analogous to the case $s=1$.  
	
	\step5 It remains to prove the identity~\eqref{equiv:cPhi}.
	We start by changing variables $r=\frac{\rho}{u}$ 
    \begin{align}
       I(\rho):=\int_0^1\frac{\rho}{u}\phi'\big(\frac{\rho}{u}\big)\frac{du}{\sqrt{1-u^2}} &= \int_\rho^1 \frac{\rho \phi'(r)}{r\sqrt{1-\tfrac{\rho^2}{r^2}}} \ud{r}. 
    \end{align}
    We further change variables $y^2 +\rho^2 = r^2$ to find
    \begin{align}
        I(\rho)&= \frac12 \int_\Reals \frac{\phi'(\sqrt{y^2+\rho^2})\rho}{\sqrt{y^2+\rho^2}} = \frac12 \int_\Reals \nabla \Phi(\rho e_1 +y e_2) \cdot e_2 \ud{y}
    \end{align}
    Hence, $I(\rho)$ is the rectilinear force evaluated along a trajectory. Due to the radial symmetry of $\Phi$ we have
    \begin{align}
        I^2(\rho) = \frac14 \left|\int_\Reals \nabla \Phi(\rho e_1 +y e_2) \ud{y}\right|^2.
    \end{align}
    Inserting this representation back into the formula for $c_\Phi$ reveals
    \begin{align}
        c_\Phi=&\int_0^1 \bigg(\int_0^1\frac{\rho}{u}\phi'\big(\frac{\rho}{u}\big)\frac{du}{\sqrt{1-u^2}}\bigg)^2\rho\,d\rho = \int_0^1 I(\rho)^2\rho\,d\rho \\
        =&\int_0^1 \frac14 \left|\int_\Reals \nabla \Phi(\rho e_1 +y e_2) \ud{y}\right|^2 \rho\,d\rho \\
        =&\frac1{8\pi} \int_{\R^3} \delta(e_1 \cdot x)  \left|\int_\Reals \nabla \Phi(x +y e_1) \ud{y}\right|^2 \,dx.
    \end{align}
    We pass to Fourier variables using Plancherel's identity
    \begin{align}
        c_\Phi=&\frac1{8\pi} \int_{\R^3} \delta(e_1 \cdot x)\left|\int_\Reals \nabla \Phi(x +y e_1) \ud{y}\right|^2 \ud{x}\\
        &=\frac1{32\pi^3}\int_{\R^3} \delta(e_1 \cdot k)  |k|^2|\hat{\Phi}(k)|^2 \ud{x}\\
        &=\frac1{16\pi^2}\int_0^\infty k^3|\hat{\Phi}(k)|^2dk,
    \end{align}
    and obtain the desired identity.
\end{proof}

\subsection{Proof of Corollary~\ref{cor:solutions}} \label{subsec:solutions}

\begin{proof}[Proof of Corollary~\ref{cor:solutions}]
	We use the well-known fact that $\LL_\e$ is a family of self-adjoint non-positive operators on ${L^2_M(\R^3 \times \R^3)}$. Indeed, symmetrizing the operator we obtain for any function $\psi\in {L^2_M(\R^3 \times \R^3)}$ 
	\begin{align}
	(\LL_\e \psi , \psi)_{L^2_M} = - &\frac14 \iint_{\R^3\times \R^3}\int_{S^2_+}  (\psi_1'+ \psi_2'-\psi_1-\psi_2)^2  M(v_1)M(v_2) (v_1-v_2)\cdot \nu \ud{v_1} \ud{v_2} \ud{\nu}\ud{x} \\
	&\leq 0.
	\end{align}
	This immediately shows our solutions are uniformly bounded by
	\begin{align}
	\|\psi_\e(t)\|_{L^2_M(\R^3 \times \R^3)} \leq 	\|\psi_\e^\circ \|_{L^2_M(\R^3 \times \R^3)}. 
	\end{align}
	By compactness, there exists a $\psi$ in $L^\infty(\R^+,{L^2_M(\R^3 \times \R^3)})$ and a subsequence $\psi_\e$, such that 
	\begin{align}
	\psi_\e \rightharpoonup^* \psi, \quad \text{in } L^\infty(\R^+,{L^2_M(\R^3 \times \R^3)}).
	\end{align}
	It remains to show that $\psi$ is the desired solution of the linearized Landau equation. To this end, let $\zeta\in C^\infty_c(\R\times\R^3 \times \R^3)$ be arbitrary. Using the self-adjointness of $\LL_\e$ we get
	\begin{align}
	&( \psi_\e(t),\zeta(t))_{L^2_M(\R^3 \times \R^3)}-( \psi^\circ,\zeta(0))_{L^2_M(\R^3 \times \R^3)} \\
	= &  \mathfrak{d}_\e^{-1} \int_0^t (\LL_\e \psi_\e(s), \zeta(s))_{L^2_M(\R^3 \times \R^3)} ds + \int_0^t ( \psi_\e(s),v \cdot \nabla_x \zeta(s))_{L^2_M(\R^3 \times \R^3)} ds \\
	= & \mathfrak{d}_\e^{-1} \int_0^t ( \psi_\e(s),\LL_\e  \zeta(s))_{L^2_M(\R^3 \times \R^3)} ds + \int_0^t ( \psi_\e(s),v \cdot \nabla_x \zeta(s))_{L^2_M(\R^3 \times \R^3)} ds .
	\end{align}
	Now we can pass to the limit. Since $\mathfrak{d}_\e^{-1}\LL_\e \zeta \rightarrow \KK \zeta$ in $L^1(\R^+;{L^2_M(\R^3 \times \R^3)})$, and $\psi_\e \rightharpoonup^* \psi$, we get
	\begin{align}
	&( \psi_\e(t),\zeta(t))_{L^2_M(\R^3 \times \R^3)}-( \psi^\circ,\zeta(0))_{L^2_M(\R^3 \times \R^3)}	 \\
	=&\int_0^t ( \psi(s),\KK  \zeta(s))_{L^2_M(\R^3 \times \R^3)}+ ( \psi_\e(s),v \cdot \nabla_x \zeta(s))_{L^2_M(\R^3 \times \R^3)} ds .
	\end{align}
	Since $\KK$ is also self-adjoint, $\psi$ satisfies the linearized Landau equation, and the claim follows. 
\end{proof}

\section{Convergence to non-cutoff Boltzmann equation for short-range potentials with \texorpdfstring{$s>1$}{s > 1}} \label{sec:hard}

\begin{proof}[Proof of Theorem~\ref{thm:2}]
	In order to simplify the presentation, fix $f(0)=1$.
	
	We define the auxilary potentials
	\begin{align}
	\Lambda(x) &= \lambda(|x|), \quad &\lambda(r)&= 1/r^s , \label{def:lambda} \\
	\Lambda_\e(x) &= \lambda_\e(|x|), \quad &\lambda_\e(r)&= f(\e r)/r^s,\label{def:lambdae} 
	\end{align}
	and observe that the following identity holds:
	\begin{align}
	\Phi_\e(x) = \Lambda_{\e^{{1}/{s}}}\bigg(\frac{x}{\e^{{1}/{s}}}\bigg).
	\end{align} 
	Due to this scale invariance, we can relate the scattering map associated to $\Phi_\e$ to the scattering map associated to $\Lambda_{\e^\frac1{s}}$. More precisely, let $X_i(t)$, $V_i(t)$, $t\in \R$, $1\leq i\leq 2$ solve
	\begin{align}
	\dot{X}_i(t) = V_i, \quad \dot{V}_i(t) = -\sum_{j\neq i}\nabla \Phi_\e (X_i-X_{i}), 
	\end{align}
	then $\tilde{X}_i(t) = \e^{-\frac1{s}}X_i(t \e^{\frac1{s}})$, $\tilde{V}_i(t)=V_i(t \e^{\frac1{s}})$ solve
	\begin{align}
	\dot{\tilde X}_i(t) = V_i, \quad \dot{\tilde V}_i(t) = -\sum_{j\neq i}\nabla \Lambda_{\e^\frac1{s}} (\tilde X_i-\tilde X_{i}).
	\end{align}
	Therefore  the scattering of particles through $\Phi_\e$ with collision parameters $(v_1,v_2,\rho,\eta^\perp)$ are up to scaling equivalent to the scattering through $\Lambda_{\e^{{1}/{s}}}$ with collision parameters $(v_1,v_2,\e^{-1/s}\rho,\eta^\perp)$. Hence, the corresponding linearized Boltzmann operators satisfy
	\begin{align}\label{eq:BMscaling}
	\e^{-2/s}\mathcal{L}_{\Phi_\e} = \mathcal{L}_{\Lambda_{\e^{{1}/{s}}}}.
	\end{align}	
	Therefore, it remains to prove that for any $\psi\in\mathcal{C}^\infty_c(\mathbb{R}^3)$, 
	\[\mathcal{L}_{\Lambda_{\e}}\psi\underset{\e\rightarrow 0}{\longrightarrow}\mathcal{L}_{\infty}\psi:=\mathcal{L}_{\Phi}\psi\]
	uniformly on $\mathbb{R}^3$.
	
	We will do so by direct comparison of the scattering problems associated to the potentials $\Lambda_\e$ and $\Phi$. 
	
	To this end, let $v_1\neq v_2$ be incoming velocities and $\rho$ an impact factor. As before we write 
	\begin{align}
	\eta := \frac{v_1-v_2}{|v_1-v_2|}.
	\end{align}
	We denote by $(v_{1,\e}',v_{2,\e}')$ the post-collisional velocities and by $\theta_\e$  the deviation angle obtained for the pair potential $\Lambda_\e$ for a given impact parameter $\rho$. Similarly, we denote by $v_1'$, $v_2'$ and $\theta$ the corresponding variables associated to the scattering via the limiting potential $\Phi$. Recall that the outgoing velocities are given by ($i=1,2$):
	\begin{equation} \label{eq:thetaRepr}
	\begin{aligned}
	v_{i}' &= \frac{v_1+v_2}{2}-(-1)^i \frac{|v_1-v_2|}{2}\left(\sin \theta \cdot  \eta^\perp+\cos \theta \cdot \eta\right),\\
	v_{i,\e}' &= \frac{v_1+v_2}{2}-(-1)^i \frac{|v_1-v_2|}{2}\left(\sin \theta_{\e} \cdot \eta^\perp+\cos \theta_{\e} \cdot \eta\right).
	\end{aligned}	
	\end{equation}
	We separate the regions	
	\begin{align}
	\mathcal{Q}_\e&:=\{(v_2,\rho)\in\mathbb{R}^3\times\mathbb{R}^+,\,|v_2-v_1|>3\e^{s/20},\,\rho<\e^{-1/10}/2\},\\
	\mathcal{P}_\e &:= \R^3 \times \R^+ \setminus \mathcal{Q}_\e. 
	\end{align}
	On the set $\mathcal{Q}_\e$, we have good control of the scattering problem, while the set $\mathcal{P}_\e$ leads to an error term. We therefore introduce separate their respective contributions to $\LL_{\Phi}$ and $\LL_{\Lambda_\e}$ by introducing the decomposition
	\begin{equation}
	\LL_{\Phi} \psi (v_1) =\LL_{\Phi,1} \psi (v_1) + \LL_{\Phi,2} \psi (v_1),
	\end{equation}
	where $\LL_{\Phi,1}$ and $\LL_{\Phi,2}$ are given by
	\begin{align*}
	\LL_{\Phi,1} \psi (v_1) &=  \frac12  \iint_{\mathcal{Q}_\e} \int_{S^1_\eta} \big(\psi(v'_{1})+\psi(v'_{2})-\psi(v_1)-\psi(v_2)\big)  \rho |v_2-v_1|M(v_2) d\eta^\perp d\rho\, dv_2,\\
	\LL_{\Phi,2} \psi (v_1) &=  \frac12  \iint_{\mathcal{P}_\e} \int_{S^1_\eta} \big(\psi(v'_{1})+\psi(v'_{2})-\psi(v_1)-\psi(v_2)\big)  \rho |v_2-v_1|M(v_2) d\eta^\perp d\rho\, dv_2.
	\end{align*}
	In the same way, we decompose $\LL_{\Lambda_\e}=\LL_{\Lambda_\e,1}+\LL_{\Lambda_\e,2}$.
	
	We begin by estimating the error terms $\LL_{\Phi,2}$, $\LL_{\Lambda_\e,2}$. 
	\begin{lemma}\label{lem:LL2}
		For $\psi\in C^2_b(\R^3)$ we have the estimate
		\begin{align}
			\|\LL_{\Lambda_\e,2} \psi\|_{L^\infty(\R^3)} + \|\LL_{\Phi,2} \psi\|_{L^\infty(\R^3)} \leq C\left(\|\nabla\psi\|+\|\nabla^2\phi\|\right)\e^{\frac{2(s-1)}{10}}.
		\end{align}
	\end{lemma}
	\begin{proof}
	Using~\eqref{eq:thetaRepr} together with the simple identity
	\begin{align}
		v_1 = \frac{v_1+v_2}{2} + \frac{|v_1-v_2|}{2} 
		\eta, \quad  v_2 = \frac{v_1+v_2}{2} - \frac{|v_1-v_2|}2 \eta,
	\end{align}
	we obtain the estimate
	\begin{align*}
		&\bigg|\int_{S^1_\eta} \big(\psi(v'_{1})+\psi(v'_{2})-\psi(v_1)-\psi(v_2)\big)  \ d\eta^\perp\bigg|\\
	\leq &\bigg|\int_{S^1_\eta} 2\nabla\psi\left(\tfrac{v_1+v_2}{2}\right)\cdot\frac{|v_1-v_2|}{2}\left(\sin \theta\eta^\perp+(1-\cos \theta)\eta\right)   \ d\eta^\perp\bigg|\\
	&+2\pi\|\nabla^2\psi\|_{L^\infty(\R^3)}\frac{|v_1-v_2|^2}{8}\left(\sin^2 \theta+(1-\cos \theta)^2\right)\\
	\leq &C\sin^2\left(\tfrac{\theta}{2}\right) \left(\|\nabla\psi\|_{L^\infty(\R^3)} |v_1-v_2|+\|\nabla^2\psi\|_{L^\infty(\R^3)} |v_1-v_2|^2\right) .
	\end{align*}
	Note that the last inequality follows from
	\begin{align}
		1-\cos(\theta) = 2 \sin^2(\tfrac{\theta}2). 
	\end{align} 
	Inserting the estimate into the definition of $L_{\Phi,2}$ yields 
	\begin{align*}&\left|\LL_{\Phi,2}\psi(v_1)\right|\leq C \big(\|\nabla\psi\|_{L^\infty}+\|\nabla^2\psi\|_{L^\infty}\big)  \iint_{\mathcal{P}_\e} \sin^2\left(\tfrac{\theta}{2}\right)\left(|v_1-v_2|^2+|v_1-v_2|^3\right)\rho M(v_2)dv_2\,d\rho\\
	&\leq C \big(\|\nabla\psi\|_{L^\infty}+\|\nabla^2\psi\|_{L^\infty}\big) \iint_{\substack{|v_1-v_2|\leq 3\e^{\frac{s}{20}}\\2 \rho<\e^{-\frac{1}{10}}}}\left(|v_1-v_2|^2+|v_1-v_2|^3\right)\rho M(v_2)dv_2\,d\rho\\
	&~+ C \big(\|\nabla\psi\|_{L^\infty}+\|\nabla^2\psi\|_{L^\infty}\big) \iint_{\substack{|v_1-v_2|^2\rho^s\leq 4 \\2 \rho\geq\e^{-\frac{1}{10}}}}\left(|v_1-v_2|^2+|v_1-v_2|^3\right)\rho M(v_2)dv_2\,d\rho\\
	&~+ C \big(\|\nabla\psi\|_{L^\infty}+\|\nabla^2\psi\|_{L^\infty}\big) \iint_{\substack{|v_1-v_2|^2\rho^s\geq 4 \\ 2\rho\geq\e^{-\frac{1}{10}}}}\sin^2\left(\tfrac{\theta}{2}\right)\left(|v_1-v_2|^2+|v_1-v_2|^3\right)\rho M(v_2)dv_2\,d\rho.
	\end{align*}
	The integrals appearing in the first two terms can be bounded by
	\begin{align}
		&\iint_{\substack{|v_1-v_2|\leq 3\e^{\frac{s}{20}}\\ 2\rho<\e^{-\frac{1}{10}}}}\left(|v_1-v_2|^2+|v_1-v_2|^3\right)\rho M(v_2)dv_2\,d\rho  \\
		+&\iint_{\substack{|v_1-v_2|^2\rho^s\leq 4 \\2 \rho\geq\e^{-\frac{1}{10}}}}\left(|v_1-v_2|^2+|v_1-v_2|^3\right)\rho M(v_2)dv_2\,d\rho \leq C \e^{(5 s-4)/20}.
	\end{align}	
	It remains to estimate the third term. The estimate 
	\begin{align}
	\Bigg|\kappa\int_0^1\tfrac{\rho}{u}\Phi'\left(\tfrac{\rho}{u}\right)\frac{du}{\sqrt{1-u^2}}\Bigg|\leq\frac{C\kappa}{\rho^s},
	\end{align} 
	together with~\eqref{eq:Kasym} imply that for $\kappa=|v_1-v_2|^{-2}$ and $K(\rho)<C/r^s$ we have
	\begin{align}
		|\theta|\leq C(\kappa/\rho^s+\kappa^2/\rho^{2s})\leq 6C\kappa/\rho^{s}.
	\end{align}
	This gives us the desired bound for the integral in the third term:
	\begin{equation*}
		\iint_{\substack{v_2\in\mathbb{R}^3\\ \rho>\e^{-\frac{1}{10}}}} \left(|v_1-v_2|^{-2}+|v_1-v_2|^{-1}\right)\rho^{1-2s} M(v_2)d\rho\,dv_2\leq C\,\e^{2(s-1)/10}.
	\end{equation*}
	Collecting all the estimates, we obtain the desired bound for $L_{\Phi,2}\psi$
	\begin{equation}
	\left|\LL_{\Phi,2}\psi(v_1)\right|\leq C\left(\|\nabla\psi\|+\|\nabla^2\phi\|\right)\e^{\frac{2(s-1)}{10}}.
	\end{equation}
	The same estimation holds for $\LL_{\Lambda_\e,2}$.
	\end{proof} 
	
	It remains to estimate the difference 
	\begin{align*}
	\LL_{\Phi_\e,1}\psi(v_1)&-\LL_{\Phi,1}\psi(v_1)\\
	&=\iint_{\mathcal{Q}_\e}\int_{S_\eta^1} \left(\psi(v'_{1,\e})-\psi(v'_{1})+\psi(v'_{2,\e})-\psi(v'_{2})\right)d\eta^\perp|v_1-v_2|\rho M(v_2)dv_2\,d\rho.
	\end{align*}
	We constructed the set $\mathcal{Q}_\e$ in such a way that we can directly compare the scattering problems associated to $\Lambda_\e$ and $\Phi$. The following Lemma shows the convergence of the deviation angles, i.e. $\theta_\e \rightarrow \theta$.
	\begin{lemma}
		For $(\rho,v_2)\in\mathcal{Q}_\e$, there exists a constant $C$ independant of $\e$ such that
		\begin{equation}\label{eq: est angle}
		\left|\theta_\e-\theta\right|\leq C\left[\e^{3/10}+\min\left(1,\frac{\e^{4/10}}{\rho^{s-1}}\right)\right].
		\end{equation}
	\end{lemma}
	\begin{proof}
		First note that $\theta$ and $\theta_\e$ are in $[0,\pi]$, so we have the trivial bound $|\theta-\theta_\e|\leq\pi$. We now give a more precise bound. To this end, we introduce the functions $F$, $F_\e$ defined by		
		\begin{align}
		F(r):=1-\frac{\rho^2}{r^2}-\frac{4}{|v_1-v_2|^2r^s},~F_\e(r):=1-\frac{\rho^2}{r^2}-\frac{4f(\e r)}{|v_1-v_2|^2r^s}.
		\end{align}
		We observe that both functions are strictly increasing and satisfy  $F_\e\geq F$. Their asymptotic values are given by
		\begin{align}
			\lim_{r\rightarrow 0} F(r) = \lim_{r\rightarrow 0} F_\e(r)=-\infty ,\quad \lim_{r\rightarrow \infty} F(r) = \lim_{r\rightarrow \infty} F_\e(r) = 1,
		\end{align}
		so they each have unique zeros $r_0$, $r_\e$ which satisfy $r_0\geq r_\e$. In light of~\eqref{eq:theta}, we can express the difference $\theta_\e-\theta$ as
		\begin{align}
			\theta_\e -\theta  = 2\int_{r_0}^\infty \frac{\rho}{r^2\sqrt{F(r)}} dr -2  \int_{r_\e}^\infty \frac{\rho}{r^2\sqrt{F_\e(r) }} dr .
		\end{align}
		
		\step{1} On the set $\mathcal{Q}_\e$ we have
		\begin{align}\label{est:rreps} 
			\e^{-\frac1{10}}>	r_0\geq r_\e  .
		\end{align}
		Indeed, this follows by monotonicity and evaluating the functions for the value $\e^{-\frac1{10}}$:
		\begin{align}
			F(\e^{-\frac{1}{10}}) \geq 1-\frac{\e^{-\frac{2}{10}}}{4\e^{-\frac{2}{10}}}-\frac{4}{9\e^{\frac{2s}{20}}\e^{-\frac{s}{10}}}\geq \frac{1}{4}>0.
		\end{align}
		
		\step{2} Let $h_\e\geq 0$ be  such that $r_\e=r_0/(1+h_\e)$. Then we have
		\begin{align}\label{est:h}
			h_\e \leq C \e^\frac9{10}.
		\end{align}
		In order to prove~\eqref{est:h}, we observe that by construction
		$F_\e\big(\tfrac{r_0}{1+h_\e}\big) = F(r_0)=0$. 
		This in turn implies 
		\begin{align*}
		 \left(\frac{\rho^2}{{r_0}^2}+\frac{4}{|v_1-v_2|^2{r_0}^s}\right)\left[(1+h_\e)^s-1\right]&=\textcolor{red}{4}\frac{1-f\big(\tfrac{\e r_0}{1+h_\e}\big)}{|v_1-v_2|^2{r_0}^s}(1+h_\e)^s.
		\end{align*}
		Since $r_0$ is a zero of $F$, and $(\rho,v_2)\in \mathcal{Q}_\e$ we obtain
		\begin{align}
			\frac{(1+h_\e)^s-1}{(1+h_\e)^s}&\leq4( 1-f\big(\tfrac{\e r_0}{1+h_\e}\big)).
		\end{align}
		Finally, $f(0)=1$ and $f$ is Lipschitz, therefore 
		\begin{align*}
		\frac{(1+h_\e)^s-1}{(1+h_\e)^{s-1}}\leq C\e
		r_0\leq C\e^{\frac{9}{10}},
		\end{align*}
		and the claim follows. 
		
%		Hence for $\e$ small enough,
%		\begin{equation}
%		\frac{\left(1+\e^{\frac{6}{10}}\right)r_0-r_\e}{r_e}=\frac{r_0}{r_\e}\left(\e^{\frac{6}{10}}+\frac{r_0-r_\e}{r_0}\right)\leq \frac{1}{1+\e^{\frac{9}{10}}}\left(\e^{\frac{6}{10}}+\e^{\frac{9}{10}}\right)\leq C\e^{\frac{6}{10}}.
%		\end{equation}
		
		\step{3} 
		There exists a constant $c>0$ such that the following inequalities for $F$ hold:
		\begin{align}\label{est:F}
			F(r) \geq \begin{cases}
				c \e^\frac6{10} \quad &\text{for $r\geq r_0(1+\e^{\frac6{10}}$)}\\
				c \quad &\text{for $r\geq 2r_0$}.
			\end{cases}
		\end{align}
		To prove this, we consider $k \geq \e^\frac6{10}$ and distinguish two cases:
		\begin{itemize}
			\item $s\leq 2$: 	Using $F(r_0)=0$ we find for some $c>0$:
			\begin{equation*}
			\begin{split}
			F(r_0(1+k))&=1-\tfrac{\rho^2}{r_0^2} \tfrac{1}{(1+k)^2}-\left(1-\tfrac{\rho^2}{r_0^2}\right)\tfrac{1}{(1+k)^s}\\
			&=\tfrac{(1+k)^s-1}{(1+k)^s}+\tfrac{\rho^2}{r_0^2}\tfrac{(1+k)^{2-s}-1}{(1+k)^2}\\
			&\geq \tfrac{(1+k)^s-1}{(1+k)^s}.
			\end{split}
			\end{equation*}
			
			\item $s>2$: similarly, we can estimate
			\begin{equation*}
			\begin{split}
			F(r_0(1+k))&=1-\left(1-\tfrac{4}{|v_1-v_2|^2r_0^{s}}\right)\tfrac{1}{(1+k)^2}-\tfrac{4}{|v_1-v_2|^2r_0^{s}}\tfrac{1}{(1+k)^s}\\
			&=\tfrac{(1+k)^2-1}{(1+k)^2}+\tfrac{4}{|v_1-v_2|^2r_0^s}\tfrac{(1+k)^{s-2}-1}{(1+k)^s}\\
			&\geq \tfrac{(1+k)^2-1}{(1+k)^2}.
			\end{split},
			\end{equation*}
		\end{itemize}
		Combining the two estimates we deduce~\eqref{est:F}.

		\step{4} For some $c>0$ independent of $\e$, $F$ and $F_\e$ satisfy:
		\begin{equation}\label{est:rclose}
		\begin{aligned}	
			F(r) 		&\geq c \frac{r-r_0}{r_0}\quad &&\text{for } r\in (r_0,(1+\e^\frac6{10}r_0)),\\
			F_\e(r) 	&\geq c\frac{r-r_\e}{r_\e} \quad &&\text{for } r\in (r_\e,(1+\e^\frac6{10}r_0)).
		\end{aligned}
		\end{equation} 
		We note that, in light of Step~1, $\e r\leq \e^{9/10}$ and thus $f(\e r)>1/2$ for $\e$ small enough. We then compute the derivatives of $F$ and $F_\e$ respectively:
		\begin{align}
			F'(r)&=\frac{1}{r}\left(2\tfrac{\rho^2}{r^2}+s\tfrac{4}{|v_1-v_2|^2r^{s}}\right) ,\\
			F_\e(r)&=\frac{1}{r} \left(2\tfrac{\rho^2}{r^2}+s\tfrac{4f(\e r)}{|v_1-v_2|^2r^{s}}+\tfrac{4\e f'(\e r)}{|v_1-v_2|^2r^{s-1}}\right).
		\end{align}
		Using the boundedness of $f'$, for $\e>0$ small enough we have the lower bound
		\begin{align}
			F_\e(r)&\geq \frac{1}{2r} \left(2\tfrac{\rho^2}{r^2}+s\tfrac{4f(\e r)}{|v_1-v_2|^2r^{s}}\right).
		\end{align}
		Moreover, on the domains in~\eqref{est:rclose} we have respectively
		\begin{align}
			r_0 \leq r \leq 2 r_0,\quad r_\e \leq r\leq 2r_\e,
		\end{align}
		so we can conclude 
		\begin{align}
			F'(r) \geq  \frac{c}{r_0}, \quad F'_\e(r) \geq \frac{c}{r_\e},
		\end{align}
		and the statement follows.

		\step{5} We can now estimate $\theta_\e-\theta$. Starting from the representation~\eqref{eq:theta}, we decompose
		\[\theta_\e -\pi  = -2\int_{r_\e}^\infty \frac{\rho dr}{r^2 \sqrt{F_\e(r)}} = -2\left( \int_{r_\e}^{(1+\e^{6/10})r_0} +\int_{(1+\e^{6/10})r_0}^\infty\right)\frac{\rho dr}{r^2 \sqrt{F_\e(r)}}.\]
		Using Step~2 and Step~4 and $\rho<r_\e$, we get the following bound for the first term:
		\begin{align}
			\int_{r_\e}^{(1+\e^{6/10})r_0}\frac{\rho dr}{r^2 \sqrt{F_\e(r)}}\leq \int_{r_\e}^{(1+\e^{6/10})r_0}\frac{\rho\,dr}{r_\e^2\sqrt{C\tfrac{r-r_\e}{r_\e}}} \leq \frac{C\rho}{r_\e}\sqrt{\frac{(1+\e^{6/10})r_0-r_\e}{r_\e}}\leq C\e^{3/10}.
		\end{align}
		The analogous estimate holds for $\theta$ and therefore  
		\begin{align*}
		|\theta-\theta_\e|&\leq C\e^{\frac{3}{10}}+ \int_{(1+\e^{6/10})r_0}^\infty \left|F_\e(r)^{-\frac{1}{2}}-F(r)^{-\frac{1}{2}}\right|\tfrac{\rho\,dr}{r^2}\\
		&\leq C\e^{\frac{3}{10}}+ \int_{(1+\e^{6/10})r_0}^\infty \tfrac{1-f(\e r)}{2r^s F(r)}\tfrac{\rho\,dr}{r^2\sqrt{F(r)}}\\
		&\leq C\e^{\frac{3}{10}}+C\frac{\e^{4/10}}{r_0^{s-1}} \int_{(1+\e^{6/10})r_0}^\infty \tfrac{\rho\,dr}{r^2\sqrt{F(r)}}\\
		&\leq C\left( \e^{\frac{3}{10}}+\frac{\e^{4/10}}{\rho^{s-1}}\right) .
		\end{align*}
		Here we have used Step~3 and the Lipschitz-continuity of $f$. 
	\end{proof}

	\begin{lemma}\label{lem:LL1}
		There exist constants $\beta,C>0$ depending only on $s$ such that
		\begin{equation}
		\left|\LL_{\Lambda_\e,1}\psi(v_1)-\LL_{\Phi,1}\psi(v_1)\right|\leq C \left(\|\nabla_v\psi\|+\|\nabla^2\psi\|\right)(1+|v_1|^3)\epsilon^\beta.
		\end{equation}
	\end{lemma}
	\begin{proof}
		We start by recalling that for $i=1,2$ we have 
		\begin{align*}
			v_{i,\e}'-v'_i = (-1)^i \frac{|v_1-v_2|}{2}\big((\sin\theta-\sin\theta_\e)\eta+(\cos\theta-\cos\theta_\e)\eta^\perp\Big).
		\end{align*}
		This allows us to obtain an estimate in terms of the difference of deviation angles:  
		\begin{align}\label{eq: esti diference}
		\bigg|\int_{S^1_\eta} \big(\psi(v'_{1,\e})-\psi(v'_{1})&+\psi(v'_{2,\e})-\psi(v'_{2})\big)  \ d\eta^\perp\bigg|\leq C\left|\sin\left(\tfrac{\theta-\theta_\e}{2}\right)\right| \|\nabla\psi\| |v_1-v_2|. 
		\end{align}
		We insert the estimate~\eqref{eq: est angle} for the deviation angles into \eqref{eq: esti diference}, and find 
		\begin{align*}
			&\Big|\LL_{\Lambda_\e,1}\psi(v_1)-\LL_{\Phi,1}\psi(v_1)\Big|\\
			\leq& C_\psi\iint_{\mathcal{Q}_\e}(|v_1-v_2|^2+|v_1-v_2|^3)\left(\epsilon^{\frac{3}{10}}+\min\left(1,\frac{\e^{4/10}}{\rho^{s-1}}\right)\right)M(v_2)\rho\,d\rho\,dv_2\\
			\leq &C_\psi \left(1+|v_1|^3\right)\left(\epsilon^{\frac{1}{10}}+\int_0^{\e^{-\frac{1}{10}}}\min\left(1,\frac{\e^{4/10}}{\rho^{s-1}}\right)\rho\,d\rho\right).
	\end{align*}
	The claim now quickly follows by distinguishing the cases $s>3$, $s=3$ and $s<3$
	\begin{itemize}
		\item 	For $s>3$,
		\begin{align*}
		\int_0^{\e^{-\frac{1}{10}}}\min\left(1,\frac{\e^{4/10}}{\rho^{s-1}}\right)\rho\,d\rho &\leq \int_0^{\e^{\frac{4}{10(s-1)}}}\rho + \int_{\e^{\frac{4}{10(s-1)}}}^\infty \frac{\e^{4/10}d\rho}{\rho^{s-2}}\\
		&\leq \frac{\e^{\frac{8}{10(s-1)}}}{2} +\frac{\e^{\frac{4}{10}\left(1-\frac{s-3}{s-1}\right)}}{s-3}\leq C\e^{\frac{4}{10(s-1)}}.
		\end{align*}
		\item 	In the same way for $s=3$, 
		\begin{align*}
		\int_0^{\e^{-\frac{1}{10}}}\min\left(1,\frac{\e^{4/10}}{\rho^{s-1}}\right)\rho\,d\rho \leq C\e^{\frac{4}{10(s-1)}}|\log\e|.
		\end{align*}
		\item For $s<3$, 
		\begin{align*}
		\int_0^{\e^{-\frac{1}{10}}}\min\left(1,\frac{\e^{4/10}}{\rho^{s-1}}\right)\rho\,d\rho &\leq \int_0^{\e^{\frac{4}{10(s-1)}}}\rho + \int_{\e^{\frac{4}{10(s-1)}}}^\infty \frac{\e^{4/10}d\rho}{\rho^{s-2}}\\
		&\leq \frac{\e^{\frac{4}{10}-\frac{3-s}{10}}}{3-s}\leq C \e^{\frac{1}{10}}.
		\end{align*}
		Collecting the three estimates concludes the proof.
	\end{itemize}
	\end{proof}

	The proof of Theorem~\ref{thm:2} is now straightforward. We recall
	\begin{align}
		(\LL_{\Lambda_{\e}} - \LL_{\Phi} )\psi =  (\LL_{\Lambda_{\e},1} - \LL_{\Phi,1} )\psi + (\LL_{\Lambda_{\e},2} - \LL_{\Phi,2} )\psi
	\end{align}
	and apply Lemma~\ref{lem:LL1} and Lemma~\ref{lem:LL2}.
\end{proof}

%\nocite{*}

	\section*{Acknowledgements}

	R.W. acknowledges support of SFB 65 "Taming Complexity in Partial Differential Systems" at the University of Vienna. Furthermore, R.W. would like to thank
	the Isaac Newton Institute for Mathematical Sciences for support and hospitality during the programme
	"Frontiers in kinetic theory: connecting microscopic to macroscopic scales - KineCon 2022" when work
	on this paper was undertaken. This work was supported by EPSRC Grant Number EP/R014604/1.
	
\bibliographystyle{plain}
\bibliography{BW22}
	
\end{document}